\documentclass[oneside,english]{amsart}
\usepackage[T1]{fontenc}
\usepackage{babel}
\usepackage{geometry}
\usepackage{prettyref}
\usepackage{float,subfig}
\usepackage{amsthm}
\usepackage{amssymb}
\usepackage{type1cm}
\usepackage{mathrsfs}
\usepackage{stmaryrd} 
\usepackage{ucs} 
\usepackage[utf8x]{inputenc}  
\PrerenderUnicode{∞}
\PrerenderUnicode{≠}
\PrerenderUnicode{²}

\usepackage[unicode=true, pdftitle={For Complex Orientations
  Preserving Power Operations, p-Typicality is Atypical},
 pdfauthor={Niles Johnson and Justin Noel},
 pdfkeywords={
   H infinity orientation,
   power operations,
   Brown Peterson spectrum,
   Complex cobordism,
 },
 bookmarks=true,bookmarksnumbered=false,bookmarksopen=false,
 breaklinks=false,pdfborder={0 0 0},backref=false,
 colorlinks=false
 ]{hyperref}
\usepackage[matrix,cmtip,arrow]{xy}

\usepackage{Definitions,Environments}

\newcommand{\maxp}{13}

\begin{document}

\title{For Complex Orientations Preserving Power Operations, $p$-typicality is Atypical}
\author{Niles Johnson}
\address{Department of Mathematics, University of Georgia, Athens, GA
  30606}
\email{njohnson@math.uga.edu}
\author{Justin Noel}
\address{
Institut de Recherche Mathématique Avancée, Université de Strasbourg,
7 rue René-Descartes
 67084 Strasbourg Cedex
}
\email{noel@math.u-strasbg.fr}
\date{07 June 2010}
\thanks{The first author was partially supported by the NSF: VIGRE Grant
DMS-0738586.}
\thanks{The second author was partially supported by the
ANR: Projet BLAN08-2\_338236, HGRT}
\begin{abstract}
  We show, for primes $p \leq \maxp$, that a number of well-known
$MU_{(p)}$-rings do not admit the structure of commutative
$MU_{(p)}$-algebras. These spectra have complex orientations that factor
through the Brown-Peterson spectrum and correspond to $p$-typical formal
group laws.  We provide computations showing that such a factorization is
incompatible with the power operations on complex cobordism.  This implies,
for example, that if $E$ is a Landweber exact $MU_{(p)}$-ring whose
associated formal group law is $p$-typical of positive height, then the
canonical map $MU_{(p)}\rightarrow E$ is not a map of $H_\infty$ ring
spectra.  It immediately follows that the standard $p$-typical orientations
on $BP,~E(n),$ and $E_n$ do not rigidify to maps of $E_\infty$ ring
spectra.  We conjecture that similar results hold for all primes.

\end{abstract}

\maketitle

\section{Introduction}\label{sec:intro}

This paper arose out of the authors' attempts to address the long-standing
open conjecture:
\begin{conjecture}
  For every prime $p,$ the $p$-local Brown-Peterson spectrum $BP$ admits
  an $E_\infty$ ring structure.
\end{conjecture}

\noindent
Quillen showed that the algebraic map $r_*\colon MU_{(p)*}\rightarrow BP_*$
classifying a universal $p$-typical formal group law over $BP_*,$ could be
realized topologically as the retraction map in a splitting of $p$-local
complex cobordism:
\begin{equation}\label{eq:splitting} 
  BP\xrightarrow{s} MU_{(p)}\xrightarrow{r} BP.
\end{equation}
This splitting plays a key role in many computational applications,
especially in the Adams-Novikov spectral sequence.

One might hope that an $E_\infty$ structure on $BP$ could be exploited
in a number of computations, in particular to prove the existence of
differentials in the Adams-Novikov spectral sequence.  In the case of
$MU,$ a very nice example of such a technique can be found in the recent
work of Hopkins, Hill, and Ravenel on the Kervaire invariant one
problem \cite{HHR09}.  In their proof, the $E_\infty$ structure on $MU$
plays a crucial role in demonstrating that certain elements in the Adams
and Adams-Novikov spectral sequences must support differentials.

A number of attempts have been made to prove the above conjecture and
there have been some positive results in this direction.  Recently,
Birgit Richter has shown that $BP$ is at least $2(p^2+p-1)$ homotopy
commutative \cite{Ric06}.  In unpublished work, Basterra and Mandell
were able to show that $BP$ admits an $E_4$-ring structure. There are
also a number of results demonstrating that $BP$ and its relatives admit
various multiplicative structures compatible with those of $MU$
\cite{EKMM97, Str99, Goe01, Laz03}.

Since Quillen's splitting plays an important role in many $BP$
computations, it is natural to ask whether either map in
\prettyref{eq:splitting} can be made into a map of $E_\infty$ ring spectra.
This splitting was shown to be $A_\infty$ in \cite{Laz04a}.  In this paper
we consider the retraction map
\[
r\colon MU_{(p)}\rightarrow BP 
\] 
and maps of spectra factoring through $r$. The section
\[s\colon BP\rightarrow MU_{(p)}\] has already been considered by Hu-Kriz-May;
they have shown that there are no $E_\infty$ ring maps whatsoever from $BP$
into $MU_{(p)}$ \cite[2.11]{HKM01} \cite[App.~B]{BaM04}.  In fact, their
proof yields the stronger result that there are no $H_\infty$ ring maps
from $BP$ into $MU_{(p)}$.

An $H_\infty$ ring spectrum can be thought of as an $E_\infty$ ring
spectrum up to homotopy; such spectra correspond to cohomology theories
with a well-behaved theory of power operations in degree 0.  To obtain
power operations in other cohomological degrees, one needs the richer
structure known as $\hi^d$.  The $\hi^2$ structure on $MU$ plays a
prominent role in this paper.  This structure arises from the $E_\infty$
structure on $BU$ via the ``Thomification'' functor \cite[IV.2]{May77}.
The resulting power operations agree with the Steenrod operations in
cobordism constructed in \cite{Die68}.

The central work of this paper is to compute the action of these power
operations on $MU_{(p)}^{2*},$ modulo the kernel of $r_*.$ These
calculations yield obstructions to lifting a ring map
\[
MU_{(p)}\rightarrow BP\rightarrow E
\]
to a map of $H_\infty$ ring spectra.

\begin{thm}\label{thm:landweber-exact-thm}
  Suppose $f\colon MU_{(p)} \to E$ is map of $H_\infty$ ring spectra
  satisfying:
  \begin{enumerate}
	\item $f$ factors through Quillen's map to $BP.$
	\item $f$ induces a Landweber 
          exact $MU_*$-module structure on $E_*.$
	\item \bf{Small Prime Condition:} $p\in \{2,3,5,7,11,\maxp\}.$ \label{enum:small-prime-cond}
  \end{enumerate}
  then $\pi_*E$ is a $\mathbb{Q}$-algebra.
\end{thm}

\noindent 
The Landweber exactness requirement is primarily a matter of convenience
for the statement and proof of this result.  Our computations have similar
consequences for a more general class of $p$-typical spectra and the proof
of \prettyref{thm:landweber-exact-thm}---given at the end of
\prettyref{sec:Main-Theorems}---illustrates how one might apply the
calculations in general.  Details of our computational methods, as well as
a more complete list of calculations, are given in
\prettyref{sec:Computer-Generated-Calculations}.

As special cases of
\prettyref{thm:landweber-exact-thm} we obtain the following:
\begin{thm}\label{thm:orientations-not-p-typical}
Suppose the Small Prime Condition holds and $n\geq 1.$ The standard
$p$-typical orientations on $E_n,$ $E(n),$ $BP\langle n \rangle,$ and $BP$
do not respect power operations.  In particular, the corresponding
$MU$-ring structures do not rigidify to commutative $MU$-algebra
structures.
\end{thm}

\begin{remark}
  The appearance of the Small Prime Condition in the above two theorems
  arises from limitations of our computational resources and the efficiency
  of our algorithms.  There is no theoretical bound on the primes for which
  our methods apply.
\end{remark}

\begin{conjecture}
  Theorems \ref{thm:landweber-exact-thm} and
  \ref{thm:orientations-not-p-typical} hold without the
  Small Prime Condition.
\end{conjecture}

In closely related work, Matt Ando \cite{And95} has constructed $\hi$
maps from $MU$ to $E_n.$ Since these maps satisfy the second condition
of \prettyref{thm:landweber-exact-thm} we have the following:

\begin{cor}
  For the primes listed above, none of the $H_\infty$ orientations on
  $E_n$ constructed in \cite{And95} are $p$-typical.
\end{cor}

\begin{notation}
  Throughout this paper we will refer to a map of ring spectra
  $MU\rightarrow E$ as a (complex) orientation on $E$ or an $MU$-ring
  structure. For convenience, we will henceforth assume all spectra are
  localized at a prime $p.$ We will also use the shorthand
  \[
  E^*\equiv E^*(*)=\pi_{-*}E
  \]
  for the $E$ cohomology of a point.  We will use cohomological gradings
  throughout this paper.
\end{notation}

\subsection{Acknowledgements} 

The authors would like to thank their advisor, Peter May, for
suggesting this problem and for his constant support, personal and
intellectual.  They sincerely appreciate the comments and suggestions that
he and Vigleik Angeltveit have made on the many drafts of this paper.

The authors would like to thank Matthew Ando for stimulating discussions
at the University of Illinois under the support of the Midwest Topology
Network.  They also wish to thank John Zekos of the University of Chicago and Jon Hanke of the University of Georgia for computer support.

\section{Main Theorems}\label{sec:Main-Theorems}

Our work studies power operations arising from an $\hi$ orientation.
This naturally produces power operations in degree 0, and our first step
is to observe that such an orientation under $MU$ defines a
wider family of power operations acting on even degrees.
\begin{thm}[See {\ref{thm:h-infty-eq-h-infty2}}]\label{thm:even-deg-power-ops}
  Suppose $f\colon MU\rightarrow E$ is a map of $H_\infty$ ring spectra, then
  for each $n\in\mathbb{Z}$ there is a map \[P_{C_p,E}\colon E^{2n}\rightarrow
  E^{2np}(BC_p)\] making the following diagram commute.  
\begin{figure}[H]
\[
\xymatrix{
MU^{2n}\ar[rr]^-{P_{C_p,MU}}\ar[d]_-{f_*}
& &
MU^{2pn}(BC_p)\ar[d]^-{f_*}\\
E^{2n}\ar[rr]^-{P_{C_p,E}} & &
E^{2pn}(BC_p)}\label{fig:h-infty2-E}
\]
\caption{Even-degree power operations induced by $\hi$ orientation
  under $MU$. \label{fig:h-infty2-internal}}
\end{figure}
\end{thm}

\noindent
These power operations are precisely tom~Dieck's Steenrod operations
in cobordism \cite{Die68}; they form what is known as an $\hi^2$ ring
structure \cite{BMMS86}.  In \prettyref{sec:thom-iso-and-h-infty-2} we provide an
alternate construction of these operations, using the Thom isomorphism
and the standard $\hi$ structure on $MU.$ In
\prettyref{thm:h-infty-eq-h-infty2} we show that an orientation
$MU \to E$ is $\hi$ if and only if it is $\hi^2,$ and
\prettyref{thm:even-deg-power-ops} follows from this.

Our applications rely on a simple observation following from
\prettyref{thm:even-deg-power-ops}: if one can find some
$x\in MU^{2n}$ such that $f_*(x)=0,$ yet $f_*P_{C_p,MU}(x)\neq 0$,
then the above square cannot commute and therefore $f$ cannot be a
map of $H_\infty$ ring spectra.  The difficulty in this approach is
showing $f_*P_{C_p,MU}(x)\neq 0$.  For well chosen $x$ and $f$ this is
a strictly algebraic problem, although not a simple one.  


\subsection{Reducing to an algebraic condition}
The theory of formal group laws provides a description of the ring
$E^*(BC_p),$ which we describe below. The ring $MU^*$ carries the universal
formal group law, and so an orientation $f\colon MU \to E$ induces a formal
group law on $E.$ For any formal group law $F,$ the $p$-series $[p]_F\xi$
and the reduced $p$-series $\left<p\right>_F \xi$ are defined by the
following equations (when clear from context, we will drop the subscript
$F$):
\[
\overbrace{\xi+_F\cdots+_F \xi}^{p \textrm{ times}}=\left[p
\right]_F \xi=\xi \cdot \left<p\right>_F \xi
\] 
The ring $E^*(BC_p)$ is
isomorphic to the quotient ring
$E^* \llbracket \xi \rrbracket / [p]\xi$. The factorization above
defines the projection map:
\begin{equation}
  q_*\colon E^* (BC_p) \cong E^*\llbracket \xi \rrbracket/[p]\xi\rightarrow
  		E^*\llbracket \xi \rrbracket/\left< p\right>\xi.\label{eq:def-q}
\end{equation}

Quillen provides a formula (\prettyref{eq:quillen1}) for
$\chi^{m+n} P_{C_p,MU}(x),$ $m \gg 0,$ when \[x=[\cp{n}]\in MU^{-2n}\]
and 
$\chi$ is defined by
\[
\chi=\prod_{i=1}^{p-1}[i]\xi \in MU^{2(p-1)}\llbracket \xi\rrbracket/[p]\xi.
\] 

\begin{prop}\label{prop:h-infty-consequence}
  Suppose $f\colon MU \to E$ is a map of $\hi^2$ ring spectra.  Then there
  are power operations $P_{C_p, E}$ making
  \prettyref{fig:h-infty2-pt} commute, and in particular
\[
f_* q_* \chi^{2n} P_{C_p, MU} [\cp{n}] = q_* \chi^{2n} P_{C_p, E} f_*[\cp{n}].
\]
\end{prop}

\begin{figure}[H]
\[
\xymatrix@C=4pc{
MU^{2*}\ar[r]^-{P_{C_p,MU}}\ar[d]_-{f_*}
& 
MU^{2p*} \llbracket \xi \rrbracket/[p]\xi\ar[d]^-{f_*} \ar[r]^-{q_* \chi^{2n}}
&
MU^{2p* + 4n(p-1)} \llbracket \xi \rrbracket/\langle
p\rangle\xi \ar[d]^-{f_*}\\
E^{2*}\ar@{-->}[r]^-{P_{C_p,E}} 
& 
E^{2p*} \llbracket \xi \rrbracket /[p] \xi \ar[r]^-{q_* \chi^{2n}}
&
E^{2p* + 4n(p-1)} \llbracket \xi \rrbracket /\langle p\rangle \xi}
\]
\caption{Complex orientations and power operations.\label{fig:h-infty2-pt}}
\end{figure}

When $E$ is $BP$ and $f=r\colon MU\rightarrow BP$ is Quillen's map, we note
that $r_*[\cp{n}] = 0$ for $n \neq p^i - 1$. By considering the
section $s \colon BP \to MU,$ James McClure showed that
\prettyref{prop:h-infty-consequence} gives a necessary and sufficient
condition for $r$ to carry an $\hi^2$ structure:

\begin{thm}[{\cite[VIII.7.7,7.8]{BMMS86}}]\label{thm:McClure-condition}
  The map $P_{C_p, BP}=r_* P_{C_p,MU} s_*$ is the only map that could
  possibly make \prettyref{fig:h-infty2-pt} commute\footnote{The
    interested reader is encouraged to verify that commutativity of
    \prettyref{fig:h-infty2-pt} does not follow formally from the
    definition of $P_{C_p, BP}$.}.  Quillen's orientation is
  $\hi^2$ if and only if the outer rectangle in this diagram commutes,
  and this occurs if and only if the elements
  \[ 
  MC_n(\xi)= r_*q_*\chi^{2n}P_{C_p,MU}[\cp{n}] \in
  BP^{2n(p-2)}\llbracket \xi \rrbracket / \left< p \right>\xi
  \] 
 are 0 when $n\neq p^i-1$ for some $i$.
\end{thm}

\noindent
In \prettyref{thm:p-typical-form} we provide an alternate formulation of
this result in the language of formal group laws.  Later, in
\prettyref{sec:Sparseness} we show that $MC_n$ is trivial when $(p-1)$ does
not divide $n$.  %

In \prettyref{sec:Computing-Obstructions} we obtain an explicit formula for
$MC_n,$ reducing our problem to algebra.  Using this formula we obtain the
calculations listed in \prettyref{sec:Computer-Generated-Calculations}
yielding \prettyref{thm:non-zero-obstructions}, the computational backbone
of the results in \prettyref{sec:intro}.  The calculations are stated in
terms of the Hazewinkel generators for
$BP^* = \bbZ_{(p)} [v_1, v_2, \ldots]$, with $v_i \in BP^{-2(p^i - 1)}$.
We note that \[r_*[\cp{p-1}] = v_1.\]
\begin{thm}\label{thm:non-zero-obstructions}\
  For $p$ satisfying the Small Prime Condition, we have the following
  expressions for
  $MC_{n} \in BP^* \llbracket \xi \rrbracket / \langle p \rangle
  \xi$.
 \begin{enumerate}
 \item When $p=2,$
   $MC_2(\xi) = \left(v_1^6+v_2^2\right) \xi ^6 + \text{higher order
     terms}$
   \subitem \hspace{1.1pc} and
   $MC_4(\xi) = v_1^4 v_2^2 \xi^{10} + \text{higher order terms}$.

  \item When $p=3,$
    $MC_4(\xi) = 2v_1^9 \xi^{22} + \text{higher order terms}$.
  \item When $p=5,$
    $MC_8(\xi) = 3v_1^{16} \xi^{88} + \text{higher order terms}$.
  \item When $p=7,$
    $MC_{12}(\xi) = 4 v_1^{22} \xi^{192} + \text{higher order terms}$.
  \item When $p=11,$
    $MC_{20}(\xi) = 9 v_1^{34} \xi^{520} + \text{higher order terms}$.
  \item When $p=13,$
    $MC_{24}(\xi) = 11 v_1^{40} \xi^{744} + \text{higher order terms}$.
 \end{enumerate}
\end{thm}

Since $BP$ carries the universal $p$-typical orientation, the results of
\prettyref{thm:non-zero-obstructions} can be applied to prove non-existence
results for other $p$-typically oriented cohomology theories.

\begin{proof}[Proof of \prettyref{thm:landweber-exact-thm}]
  Let $f\colon MU \to E$ be a map of $\hi$ ring spectra as in
  \prettyref{thm:landweber-exact-thm}, and let $p$ be a prime
  satisfying the Small Prime Condition.  

  For $n \neq p^i - 1,$ we must have $f_*[\cp{n}] = 0,$ since $f$
  factors through $r.$  By assumption, \prettyref{fig:h-infty2-pt}
  commutes so $f_*(MC_n(\xi))$ must be 0 in the quotient ring.
  Equivalently, in $E^*\llbracket\xi\rrbracket$ we have
  \[
  f_*(MC_n(\xi))=g(\xi) \cdot \langle p \rangle \xi
  \] 
  for some $g(\xi)$. Since $\langle p \rangle \xi= p+\xi(\cdots)$ we see
  that the leading coefficient of $f_*(MC_n(\xi))$ must be divisible by
  $p.$ Combining this fact with the calculations in
  \prettyref{thm:non-zero-obstructions}, we will finish the proof by
  showing that $E^*/p$ must be 0 and hence $E^*$ must be a
  $\mathbb{Q}$-algebra.

  For $p = 2$, the calculation for $MC_2$ shows $f(v_2)^2 = f(v_1)^6$ in
  $E^*/2$.  Combining this with the calculation for $MC_4$ shows
  $f(v_1)^{10}=0$ in $E^*/2.$ However, by Landweber exactness,
  multiplication by $f(v_1)$ is an injection on $E^*/2,$ so $E^*/2=0.$

  For $p > 2$ the computation for $MC_{2(p-1)}$ implies that $f(v_1)$ is
  nilpotent in $E^*/p.$ The same argument then shows that $E^*/p=0.$
\end{proof}

\section{\texorpdfstring{$\ei$ and $\hi$ Ring Spectra}{E∞ and H∞ Ring Spectra}}\label{sec:hi-and-ei-ring-spectra}

Let $\mathpzc{S}$ denote the Lewis-May-Steinberger category of
coordinate-free spectra and $\mathpzc{hS}$ the stable homotopy category. 

A spectrum in this category is indexed by finite-dimensional subspaces
of some countably infinite-dimensional real inner product space
$\mathcal{U}$. Let $\pi$ be a subgroup of $\Sigma_n,$ the symmetric group
on $n$ letters.  The space of linear isometries
$\mathcal{L}(\mathcal{U}^n,\mathcal{U})$ is a free contractible
$\Sigma_n$-space and by restriction a free contractible $\pi$-space
which we will denote $E\pi$.

For each subgroup $\pi$ of $\Sigma_n$ there is an extended power
functor on unbased spaces, based spaces, and spectra.  For an unbased
space $Z$, a based space $W$, and a spectrum $X$, the definitions are
\begin{align*}
  D_\pi Z &= E\pi  \times_\pi  Z^{\times n}\\
  D_\pi W &= E\pi_+ \wedge_\pi W^{\wedge n}\\
  D_\pi X &= E\pi \ltimes_{\pi}X^{\wedge n}.
\end{align*}
where $\ltimes$ is the twisted half-smash product of \cite{LMS86}.  The
functor from unbased to based spaces given by adjoining a disjoint
basepoint relates the extended cartesian power on unbased spaces and the
extended smash power on based spaces.  For an unbased space $Z$, there is a
homeomorphism of based spaces,
\[
D_\pi (Z_+)  \cong  ( D_\pi Z )_+.
\]
We will be using power operations on unreduced cohomology theories; as a
consequence we will focus on unbased rather than based spaces.  The
extended Cartesian power on unbased spaces is related to the extended
smash power on spectra by the following: For an unbased space $Z$
\begin{equation}\label{eq:SigmaD-equals-DSigma}
  D_\pi  \Sigma^\infty_+ (Z)  =
  D_\pi  \Sigma^\infty (Z_+)  \cong
  \Sigma^\infty  D_\pi  (Z_+) \cong
  \Sigma^\infty  (D_\pi Z )_+ =
  \Sigma^\infty_+  D_\pi Z.
\end{equation}
With \prettyref{eq:SigmaD-equals-DSigma} in mind, we may implicitly
apply the functor $\Sigma^\infty_+$ and will use the notation $D_\pi Z$ to
denote either an unbased space or a spectrum, as determined by
context.


\begin{defn}
  Let $D$ be the functor on $\mathpzc{S}$ such that
\[DX=\bigvee_{n\geq0}D_{\Sigma_n} X.\]
\end{defn}

The following result is standard (for example, see \cite{Rez98}).

\begin{prop}
There are natural transformations 
\begin{gather*}
  \mu\colon D^2\rightarrow D\\
  \eta\colon Id\rightarrow D
\end{gather*} 
that make $D$ a monad 
on $\mathpzc{S}$.
\end{prop}

\begin{defn}
  The category of $E_\infty$ ring spectra is the category of
  $D$-algebras in $\mathpzc{S}$.
\end{defn}

\begin{prop}
  The monad $D$ on $\mathpzc{S}$ descends to a monad $\widetilde{D}$ on
  the stable homotopy category $\mathpzc{hS}$.
\end{prop}
\begin{proof}
  In \cite{LMS86} it is shown that this functor preserves
  homotopy equivalences between cell spectra and takes cellular spectra
  to cellular spectra.  It follows that $D$ has a
  well-defined functor on the stable homotopy category, modeled by
  cellular spectra with homotopy classes of maps and that the
  structure maps above pass to the stable category.
\end{proof}

\begin{defn}\label{def:hi-ring-spectra}
  The category of $H_{\infty}$ ring spectra is the category of
  $\widetilde{D}$-algebras in $\mathpzc{hS}$.
\end{defn}

\begin{prop}\label{prop:s-algebras-are-h-infty}
  Let $\Gamma\colon \mathpzc{S} \to \mathpzc{hS}$ denote the canonical
  functor.  If $X$ is an $E_\infty$ ring spectrum, then $\Gamma X$ is an
  $H_{\infty}$ ring spectrum.
\end{prop}

\begin{remark}
  Nearly all known $H_\infty$ ring spectra arise by applying $\Gamma$ to
  an $E_\infty$ ring spectrum.
  In \cite{Noe09} the second author provides an example of one that is
  not.  
\end{remark}

\begin{defn}\label{def:external-power-op}
  Suppose $X$ is a spectrum, $E$ is an $H_\infty$ ring spectrum,  and
  $f\colon X\rightarrow E$ is a map representing a cohomology class in
  $E^0(X)$.  Define the $\pi\text{th}$ external cohomology operation
  \[\mathcal{P}_{\pi,E}\colon E^0(X)\rightarrow E^0(D_{\pi}X)\] by 
  \[ (X\xrightarrow{f} E) \mapsto (D_\pi X\xrightarrow{D_\pi f} D_\pi
  E\rightarrow D_{\Sigma_n} E \hookrightarrow DE\xrightarrow{\mu} E).\]
\end{defn}

If $Y$ is a space, $Y^{\times n}$ is equipped with the $\pi$ action
induced by the inclusion $\pi\rightarrow \Sigma_n$.  Regarding $Y$
as a trivial $\pi$-space, the diagonal map \[\Delta\colon Y\rightarrow
Y^{\times n}\] is $\pi$-equivariant.

\begin{defn}\label{def:internal-power-op}
  Suppose $Y$ is a space and $E$ is an $H_\infty$ ring spectrum. Define 
  $\delta\colon B\pi\times Y\rightarrow D_\pi Y$ as the following
  composite:
  \begin{equation*}
	\delta\colon (B\pi\times Y) \simeq E\pi \times_\pi Y 
	  \xrightarrow{E\pi\times\Delta} 
  		E\pi \times_\pi Y^n \cong 
		D_\pi Y.
  \end{equation*}

  Define the $\pi\text{th}$ internal cohomology operation
  $P_{\pi,E}\colon E^0(Y)\rightarrow E^0(B\pi\times Y)$ as the composite
  \[ E^0(Y)\xrightarrow{\mathcal{P}_{\pi,E}} E^0(D_\pi
  Y)\xrightarrow{\delta^*} E^0(B\pi \times Y) .\]
\end{defn}

\begin{notation}
We will drop the subscript $E$ from the power operations
$\mathcal{P}_{\pi, E}$ and $P_{\pi, E}$, when it is clear from the context.
\end{notation}

\subsection{\texorpdfstring{$H_\infty^d$ ring spectra}{H∞\^{}d ring
    spectra}}

An $\hi^d$ ring structure on a spectrum $E$ is a compatible family of
maps
\[
D_{\Sigma_n}\Sigma^{di} E \to \Sigma^{din} E
\]
for all $i \in \mathbb{Z}$ \cite[I.4.3]{BMMS86}.  When $i = 0$, these
maps define an $\hi$ structure on $E$, so every $\hi^d$ ring
spectrum is an $\hi$ ring spectrum.  The compatibility conditions are
graded analogs of those for an $\hi$ ring spectrum, and an $\hi^d$
structure on $E$ determines an $\hi$ structure on the infinite
wedge\footnote{ We note that the argument for the converse to this
statement, given in \cite[II.1.3]{BMMS86}, is incorrect.  We were
unable to find a proof for the converse to hold in this generality.}
\[
\bigvee_{i \in \mathbb{Z}} \Sigma^{di} E.
\]
Maps of $\hi^d$ ring spectra are those which commute with the family
of structure maps and so
the category of $H_\infty^d$ ring spectra is a subcategory of the
category of $H_\infty$ ring spectra.

Suppose $Y$ is a space and $E$ is an $H_\infty^d$ ring spectrum. For
each $\pi \leqslant \Sigma_n$ and for each integer $i$,  we have the
following power operations: 
\begin{align*}
  \mathcal{P}_{\pi,E}&\colon E^{di}(Y)\rightarrow E^{din}(D_\pi
  Y)\\
  P_{\pi,E}&\colon E^{di}(Y)\rightarrow E^{din}(B\pi\times Y).
\end{align*}
When $i = 0$, these maps are simply the above power operations defined
using the underlying $\hi$ structure on $E$.

\subsection{\texorpdfstring{The Thom isomorphism and $H_\infty^2$ orientations}{The Thom isomorphism and H∞² orientations}}\label{sec:thom-iso-and-h-infty-2}
Let $V_k$ denote the standard representation of $\Sigma_k$ on
$\mathbb{C}^k$ and $B\Sigma_k^{V_k\otimes \mathbb{C}^i}$ be the Thom
spectrum of the complex vector bundle $V_k\otimes\mathbb{C}^i$ over
$B\Sigma_k.$ Recall \cite[Ch.~X]{LMS86} that 
\begin{equation} \label{eq:thom-iso}
  D_{\Sigma_k} S^{2i}\cong B\Sigma_k^{V_k\otimes \mathbb{C}^i}.
\end{equation}

Since $V_k\otimes\mathbb{C}^i$ is a complex vector bundle, for any complex
oriented cohomology theory $E$ we have a Thom
isomorphism \[E^*(\Sigma^{2ki}B\Sigma_k) \cong
E^*(B\Sigma_k^{V_k\otimes \mathbb{C}^i}).\]  
Taking $\mu_{i,k}$ to be a map representing the Thom class, the Thom
isomorphism yields the following commutative diagram.  The horizontal map
is induced by the natural inclusion $S^{2ki} \to D_{\Sigma_k}S^{2i}$
and $e$ is the unit $S \to E$.
\[\xymatrix{ 
S^{2ki}\ar[rr]\ar[dr]_-{\Sigma^{2ki}e} & & D_{\Sigma_k} S^{2i}\ar[ld]^-{\mu_{i,k}}\\
& \Sigma^{2ki}E. & } \]
Note that although the Thom classes $\mu_{i,k}$ clearly depend on the
cohomology theory $E,$  we will abuse notation and use the same symbol
regardless of the cohomology theory.  

When $E=MU,$ McClure shows \cite[VII]{BMMS86} that the $\mu_{i,k}$
combine with the $\hi$ structure maps \[\mu_k\colon D_{\Sigma_k}
MU\rightarrow DMU \xrightarrow{\mu} MU\] to define an $H_\infty^2$
structure for $MU$: The structure maps are those given by the top
horizontal composite in \prettyref{fig:h-infty2-orientations}.

\begin{figure}[H]
\[
\xymatrix{
D_{\Sigma_k}(\Sigma^{2i}MU)\ar[r]\ar[d]_-{D_{\Sigma_k}(f)} &
D_{\Sigma_k} S^{2i}\wedge D_{\Sigma_k}
MU\ar[rr]^-{\mu_{i,k}\wedge \mu_k}\ar[d]_-{D_{\Sigma_k} S^{2i}\wedge f} & &
\Sigma^{2ki}MU\wedge MU \ar[r]\ar[d]^-{\Sigma^{2ki}f\wedge f}&
\Sigma^{2ki}MU\ar[d]^-{\Sigma^{2ki}f}\\
D_{\Sigma_k}(\Sigma^{2i}E)\ar[r] & D_{\Sigma_k} S^{2i}\wedge
D_{\Sigma_k}
E\ar[rr]^-{\mu_{i,k}\wedge \mu_k} & & \Sigma^{2ki}E\wedge E
\ar[r]& \Sigma^{2ki}E.
}
\]
\caption{$H_\infty^2$ orientations.\label{fig:h-infty2-orientations}}
\end{figure}

In this way the Thom isomorphism for complex oriented theories gives an
equivalence between $\hi$ orientations and $\hi^2$ orientations.

\begin{thm}\label{thm:h-infty-eq-h-infty2}
  An orientation $MU\rightarrow E$ is $H_\infty$ if and only if it is
  $H_\infty^2$.
\end{thm}
\begin{proof}
  By neglect of structure every $\hi^2$ orientation is $\hi$.  Consider an
  $H_\infty$ complex orientation $f\colon MU\rightarrow E$.
  \prettyref{fig:h-infty2-orientations} is induced by this structure and
  the left and right squares in this diagram commute for any orientation on
  $E$. The center square is the smash product of the following two squares:
  \begin{equation*}
    \xymatrix{
      D_{\Sigma_k} S^{2ik}\ar[r]^{\mu_{i,k}}\ar@{=}[d]&
      \Sigma^{2ki}MU\ar[d]_{\Sigma^{2ki}f} & 
      & D_{\Sigma_k} MU\ar[r]^-{\mu_k}\ar[d]^-{f} & MU\ar[d]^-{\Sigma^{2ki}f}\\
      D_{\Sigma_k} S^{2ik}\ar[r]^{\mu_{i,k}} & \Sigma^{2ki}E & 
      & D_{\Sigma_k} E\ar[r]^-{\mu_k} & E
    }
  \end{equation*}
  The left square commutes since $f$ sends $MU$-Thom classes to $E$-theory
  Thom classes.  The right square commutes since $f$ is an $H_\infty$ ring
  map.
  
  It follows that the center square and therefore the entire diagram
  commutes in \prettyref{fig:h-infty2-orientations}.  Another elementary
  diagram chase, using the $\hi^2$ structure of $MU$, shows that the bottom
  horizontal composite defines an $H_\infty^2$ structure on $E$.
\end{proof}

\section{The Formal Group Law Condition}\label{sec:A-Formal-Group-Theoretic-Condition}

\subsection{Formal group laws}
We recall some well-known facts about complex-oriented
cohomology theories and formal group laws (for example, see \cite[Part II]{Ada95} or
\cite{Rav00}).

\begin{defn}
  A (commutative, $1$-dimensional) formal group law $F$ over a commutative
  ring $k$ is a connected bicommutative, associative, topological Hopf
  algebra $\mathcal{A}$ with a specified isomorphism $\mathcal{A}\cong
  k\llbracket x\rrbracket$.

\end{defn}
By forgetting the grading, a graded Hopf algebra of the above form is
a formal group law.  For such Hopf algebras the completed tensor product
provides the following isomorphism:
\[ \mathcal{A}\widehat{\otimes}\mathcal{A}
	  \cong k\llbracket x_1, x_2 \rrbracket.\]

\begin{defn}
  Given a ring map $f\colon k\rightarrow k^\prime$ and a formal group law
  $\mathcal{A}$ over $k,$ the push-forward of $\mathcal{A}$
  along $f$ is the formal group law $\mathcal{A}\widehat{\otimes}_k^f k^\prime$ over
  $k^\prime$.
\end{defn}

One can formally define a ring $L$ and a formal group law $\mathcal{A}$
over $L$ such that 
\begin{align} 
  \mathpzc{Ring}(L,k) &\cong \text{Formal group laws over } k\\
  f &\rightarrow \mathcal{A}\widehat{\otimes}_L^f k 
\end{align}

%
\begin{notation}
  We will identify a formal group law $F$ with the formal power series:
  \[ x_1+_F x_2 = \Delta(x)\in k\llbracket x_1, x_2 \rrbracket.\]
\end{notation}

\begin{defn}
  Given a commutative ring $k,$ we formally adjoin the $q\text{th}$ roots of
  unity. A formal group law $F$ over $k$ is $p$-typical, if for all
  primes $q\neq p,$ the formal sum over the $q\text{th}$ roots of unity 
  \vspace{3pt}
  \[
  \sum_{\zeta^q=1}\hspace{-6pt}{\phantom{|}}^{F}\  \zeta x \]
  is trivial.
\end{defn}

\subsection{Connection to complex orientations}\label{sec:conn-cpx-or}
Recall that if $X$ is a space and $E$ is a spectrum, the function
spectrum \[E^X=F(\Sigma^\infty_+X,E)\] defines a cohomology theory
satisfying 
\begin{equation}\label{eq:function-spectrum-cohomology-theory}
  E^{X,*}(Y)\cong E^*(X\times Y),
\end{equation}
for every space $Y.$
Moreover, if $E$ admits the structure of a ring spectrum (or an
$H_\infty$ ring spectrum) then
so does $E^X$.

\begin{prop}[{\cite[3.1]{Lan76}}]\label{prop:MU-BP-BCp}
  The spectra $MU^{BC_p}$ and $BP^{BC_p}$ are ring spectra
  satisfying the following natural isomorphisms:
  \begin{align*}
	MU^{BC_p,*}X &\cong MU^*(BC_p)\widehat{\otimes}_{MU_*}MU^*(X)\\
	BP^{BC_p,*}X &\cong BP^*(BC_p)\widehat{\otimes}_{BP_*}BP^*(X).
  \end{align*}
\end{prop}

In complex cobordism there is a tautological element $x$ giving an
isomorphism
\begin{align*}
  MU^{*}(\mathbb{C}P^{\infty}) &\cong MU^{*}\llbracket x
  \rrbracket,
  \intertext{and we fix an element $\xi$ such that} 
 MU^*(BC_p) &\cong MU^*\llbracket \xi \rrbracket / [p]\xi.
  \intertext{Hence we have}
  MU^{BC_{p},\,*}(\mathbb{C}P^{\infty}) &\cong
  MU^{*}\llbracket \xi,\: x \rrbracket /[p]\xi.
\end{align*}
An orientation $f\colon MU\rightarrow E$ fixes generators $x$ and $\xi$ in
$E$-cohomology that define analogous isomorphisms.

The above tautological isomorphism in complex cobordism combined with
the multiplication on $\cpinf$ classifying a tensor product of line
bundles defines a formal group law over $MU^*$. An orientation
$MU\rightarrow E,$ induces a map $MU^*\rightarrow E^*$ which defines a
formal group law structure (also denoted by $E$) on $E^*(\cpinf)$ by
pushing forward the formal group law on $MU,$ or equivalently
\cite[II.4.6]{Ada95}, by fixing the generator $x\in E^*(\cpinf)$ above.

\begin{thm}[{\cite{Qui69a}}] \label{thm:quillen-mu}
  The map \[L\rightarrow MU^*\] classifying the tautological formal group law over $MU^*$ is an isomorphism. 
\end{thm}

Rationally, we can describe this isomorphism explicitly in terms of
the cobordism classes \[[\cp{n}] \in MU^{-2n}.\]

\begin{prop} \label{prop:Q-gens-of-MU}
  There is an algebra isomorphism
  \[
  MU^*\otimes \mathbb{Q}\cong \mathbb{Q}[[\cp{1}],[\cp{2}],\dots].
  \]
\end{prop}

With these choices, the power operation 
\[
  P_{C_{p},MU} \colon MU^{2*}(\cp{\infty}) \to MU^{BC_{p},\,2p*}(\cp{\infty})
\] 
of \prettyref{fig:h-infty2-internal} on the generator $x$ is given 
by the following formula \cite[Prop.~3.17]{Qui71d}:
\begin{equation}
P_{C_{p},MU}(x)=\prod_{i=0}^{p-1}\bigl([i]\xi\,+_{MU}\, x\bigr).
\end{equation}
Of course, after applying an orientation $f\colon MU\rightarrow E$ we obtain
\begin{equation}\label{eq:power-op-formula}
f_* P_{C_{p},MU}(x)=\prod_{i=0}^{p-1}\bigl([i]\xi\,+_{E}\,
x\bigr).
\end{equation}

Considering \prettyref{eq:power-op-formula} as a power series in $x$
whose coefficients are power series in $\xi,$ we define
\[
a_i \equiv a_i(\xi) \in E^{2(p-i-1)}(BC_p)\cong
E^{2(p-i-1)}\llbracket\xi\rrbracket/[p]\xi, ~\text{ for }~ i \ge 0
\] 
by the following expansion:
\begin{equation}\label{eq:def-aseries}
  f_* P_{C_{p},MU}(x)=a_{0}x + a_{1} x^2 + a_{2} x^3
  +\cdots.
\end{equation}
By pulling back along the inclusion \[S^2\cong\cp{1}\rightarrow
\cpinf,\]
and applying the $C_p$ analogue of \prettyref{eq:thom-iso} we see that
$a_0 x$ is the Euler class of the regular representation of $C_p$ and 
\begin{equation}\label{eq:def-chi}
  a_{0}=\chi,
\end{equation} 
is the Euler class of the \emph{reduced} regular representation of
$C_p$.  

The next result follows immediately from \prettyref{prop:additivity}.
\begin{prop} \label{prop:power_op_is_a_ring_map} Let $X$ be a
  topological space and let \[\overline{P_{C_p}}\colon
  MU^{2*}(X)\rightarrow MU^{BC_p,2*}(X)[\chi^{-1}]\] be the map which
  in degree $2n$ is $P_{C_p}/\chi^n$. Then $\overline{P_{C_p}}$ and
  $r_*\overline{P_{C_p}}$ are maps of graded rings.
\end{prop}


 Using this result and the discussion preceding
 \prettyref{thm:quillen-mu} we see that the maps
 $\overline{P_{C_{p},MU}}$ and $r_*\circ \overline{P_{C_p,MU}}$ 
 define formal group laws
 $\mathcal{UP}$ and $\mathcal{VP}$ over $MU^{BC_{p}}[\chi^{-1}]$ and
 $BP^{BC_{p}}[\chi^{-1}]$ respectively. 

\begin{figure}[ht!]
\[
\xymatrix{
MU^{2*}(\mathbb{C}P^{\infty})\ar[rr]^-{\overline{P_{C_{p},MU}}}\ar[d]_{r_*}
& &
MU^{BC_{p},2*}[\chi^{-1}](\mathbb{C}P^{\infty})\ar[d]_{r_*}\\
BP^{2*}(\mathbb{C}P^{\infty})\ar@{-->}[rr]^-{\overline{P_{C_{p},BP}}} & &
BP^{BC_{p},2*}[\chi^{-1}](\mathbb{C}P^{\infty})}
\]
\caption{A formal group theoretic condition.\label{fig:red_power_op_diagram}}
\end{figure}
\nopagebreak[4]

\begin{thm}\label{thm:p-typical-form}
  The map $r\colon MU\rightarrow BP$ is a map of $H_{\infty}$ ring spectra
  if and only if $\mathcal{VP}$ is $p$-typical. 
\end{thm}
\begin{proof}
  Since the map $r$ is a $p$-universal orientation of $BP$, there
  exists a map \[P\colon BP\rightarrow BP^{BC_p}[\chi^{-1}].\] that makes
  \prettyref{fig:red_power_op_diagram} commute if and only if
  $\mathcal{VP}$ is $p$-typical.  This happens if and only if the
  indecomposables in $MU^{-2n}$ map to zero under
  $\overline{P_{C_p,MU}}$ when $n\neq p^i-1$. Since the cobordism
  classes $[\mathbb{C}P^n]$ are rationally polynomial generators and
  all rings in sight are torsion-free, we see that $\mathcal{VP}$ is
  $p$-typical if and only if the elements $MC_n$ of
  \prettyref{thm:McClure-condition} map to 0.
\end{proof}

\section{Computing the Obstructions}\label{sec:Computing-Obstructions}

Before proving \prettyref{prop:MCn-rearranged} we will need some notation.
\subsection{Notation}\label{sec:notation}
Throughout this paper, the symbol 
\begin{equation}
  \alpha=(\alpha_0,\alpha_1,\dots)
\end{equation}
with $\alpha_n=0$ for $n\gg 0,$ will be a multi-index beginning with
$\alpha_0$.  

As the reader will see, it will also be convenient to have notation
for multi-indices starting with $\alpha_1$, so we set
\begin{equation}
  \overline{\alpha} = (\alpha_1, \alpha_2, \ldots).
\end{equation}
Given an infinite list of variables $a_0,a_1,a_2,\dots,$ we set 
\begin{equation}
  a^\alpha = a_0^{\alpha_0}a_1^{\alpha_1}\dots\quad\mathrm{and}\quad
  a^{\overline{\alpha}}=a_1^{\alpha_1}a_2^{\alpha_2}\dots.
\end{equation}
  
For any integer $n$ we define the modified multinomial coefficient
$\mu(n;\overline{\alpha})$ by the formal power series expansion:
\begin{equation}
  (1+b_1+b_2\cdots)^n=\sum_{\overline{\alpha}} 
  \mu(n;\overline{\alpha})b^{\overline{\alpha}}.\label{eq:def-mu}
\end{equation}
We also set
\begin{align}
  |\alpha|&=\sum_{i\geq 0} \alpha_i,\\
  |\alpha|^\prime&=\sum_{i\geq 0} i \alpha_i = |\overline{\alpha}|'.
\end{align}

Given a formal power series $S(z)$, let 
\begin{equation}
  S(z)[z^k]=\textrm{coefficient of } z^k \textrm{ in }S(z).
\end{equation}

\subsection{Additive and multiplicative operations}\label{sec:additive-and-multiplicative-ops}

Recall that the Landweber-Novikov algebra is the subalgebra
of $MU^*MU$ whose elements define additive cohomology operations.
This algebra is a free $\mathbb{Z}_{(p)}$-module on elements
\begin{equation}
  s_{\alpha_1,\alpha_2,\dots}=s_{\overline{\alpha}}
\end{equation}
dual to the standard basis
\begin{equation} 
  t_1^{\alpha_1}t_2^{\alpha_2}\dots=t^{\overline{\alpha}}\in 
   MU_{2|\overline{\alpha}|^\prime}MU\cong 
   MU_{2|\overline{\alpha}|^\prime}BU.
\end{equation}

To simplify our formulas we extend the indexing to multi-indices
starting with $\alpha_0$ by setting
\begin{equation}
  s_\alpha\equiv s_{\overline{\alpha}}\in MU^{2|\alpha|^\prime}MU.
\end{equation}

\begin{prop}[{\cite{Qui71d}}]\label{prop:additivity}
  If $x\in MU^{-2q}(X)$ and $m\gg 0$ then 
  \begin{equation}
    \chi^{m+q}P_{C_p}x = \sum_{|\alpha|= m}
    a^{\alpha} s_\alpha(x).\label{eq:quillen1}
  \end{equation}
\end{prop}

Since the right-hand side of \prettyref{eq:quillen1} is additive in
$x$ and $P_{C_p}$ is always multiplicative, we obtain
\prettyref{prop:power_op_is_a_ring_map} by inverting $\chi$.

For any complex oriented cohomology theory $E,$
\[[i]\xi +_E x \equiv i\xi \mod{x}, \] which implies 
\begin{equation}\label{eq:chi-non-zero-divisor}
  \chi = a_0 \equiv (p-1)!\xi^{p-1} \mod{\xi^p}.
\end{equation}

It follows that inverting $\chi$ factors through inverting $\xi,$ so
when $E$ is $MU$ or $BP,$ we have 
\[ 
E^{BC_p,*}(X)[\chi^{-1}]\cong
E^*(X)\llbracket\xi\rrbracket[\chi^{-1}]/[p]\xi\cong
E^*(X)\llbracket\xi\rrbracket[\chi^{-1}]/\langle p\rangle\xi.
\]

Since 
\[
\langle p \rangle\xi=[p]\xi/\xi \equiv p \mod{\xi}
\] 
and $(p-1)!$ is not divisible by $p$, $q_* \chi$ is not a zero-divisor.  It
follows, when $E=MU$ or $BP,$ that the localization map
\[
E^*(X)\llbracket\xi\rrbracket/\langle p\rangle \xi \rightarrow
E^*(X)\llbracket\xi\rrbracket[\chi^{-1}]/\langle p\rangle\xi
\]
is an injection.  
Applying \prettyref{prop:power_op_is_a_ring_map} proves
the following:
\begin{prop} 
  The composites \begin{align*}
	q_*P_{C_p}\colon &MU^*(\cpinf)\rightarrow MU^{BC_p,*}(\cpinf)/\langle p
	\rangle \xi\\
	r_*q_*P_{C_p}\colon &MU^*(\cpinf)\rightarrow BP^{BC_p,*}(\cpinf)/\langle p
	\rangle \xi
  \end{align*} are ring maps.
\end{prop}

\subsection{\texorpdfstring{Derivation of $MC_n$}{Derivation of MC\_n}}\label{sec:Derivation-MCn}
We begin with the following refinement of \prettyref{eq:quillen1}:
\begin{lem}\label{lem:quillen-m=n}
  \begin{equation}\label{eq:quillen1-cpn}
  \chi^{2n}P_{C_p}[\cp{n}] = \sum_{|\alpha|= n}
  a^{\alpha} s_\alpha[\cp{n}].
  \end{equation}
\end{lem}
\begin{proof}
  By \prettyref{eq:quillen1}, for $k\gg0$ we have
\begin{align*}
  \chi^{2n+k}P_{C_p}[\cp{n}] &= \sum_{|\alpha|= n+k} a^{\alpha}
  s_\alpha[\cp{n}] \\
		&=\sum_{\alpha_0=0}^{n+k}\sum_{|\overline{\alpha}|=
		n+k-\alpha_0} a_0^{\alpha_0} a^{\overline{\alpha}}
		s_{\overline{\alpha}}[\cp{n}] \\
		&=\sum_{\alpha_0=0}^{k-1}a_0^{\alpha_0}\sum_{|\overline{\alpha}|=
		n+k-\alpha_0} a^{\overline{\alpha}}
		s_{\overline{\alpha}}[\cp{n}] +
		\sum_{\alpha_0=k}^{n+k}a_0^{\alpha_0}\sum_{|\overline{\alpha}|=
		n+k-\alpha_0} \hspace{-11pt}a^{\overline{\alpha}} s_{\overline{\alpha}}[\cp{n}] 
\end{align*}

Since $MU^*$ is concentrated in non-positive degrees,
\[
s_{\overline{\alpha}}([\mathbb{C}P^n]) \in
MU^{2|\overline{\alpha}|^\prime-2n} = 0 \] when
$|\overline{\alpha}|'>n$.

In the first sum of the last equation, $|\overline{\alpha}|>n$.
Since \[|\overline{\alpha}|'=\sum_{i\geq 1}i \alpha_i \geq
\sum_{i\geq 1}\alpha_i=|\overline{\alpha}|,\]  all terms in the
first sum are trivial.  This leaves us with
\begin{align*}
   \chi^{2n+k}P_{C_p}[\cp{n}] &=
   \sum_{\alpha_0=k}^{n+k}a_0^{\alpha_0}\sum_{|\overline{\alpha}|=
   n+k-\alpha_0} a^{\overline{\alpha}} s_{\overline{\alpha}}[\cp{n}] \\
	  &=a_0^k\sum_{\alpha_0=0}^{n}a_0^{\alpha_0}\sum_{|\overline{\alpha}|=
	  n-\alpha_0} a^{\overline{\alpha}} s_{\overline{\alpha}}[\cp{n}] \\
	  &=a_0^k\sum_{|\alpha|= n} a^{\alpha} s_{\alpha}[\cp{n}]. 
\end{align*}
Since $a_0 = \chi$ is a not a zero-divisor the lemma follows.
\end{proof}

\begin{thm}[{\cite[I.8.1]{Ada95}}]\label{thm:adams-formula}
  \begin{equation}\label{eq:adams-formula}
	s_\alpha[\cpn]=\mu(-(n+1);\overline{\alpha})[\cp{n-|\alpha|^\prime}]
  \end{equation}
\end{thm}

We combine Equations \ref{eq:quillen1-cpn} and \ref{eq:adams-formula}
and obtain:

\begin{thm}\ \label{thm:McClure-formula}
  \begin{gather*}
	MC_{n}(\xi)\equiv r_* q_* \chi^{2n}
	P_{C_p}[\cpn]=\sum_{|\alpha|=n}\mu(-(n+1);\overline{\alpha})\,\,r_{*}[\cp{n-|\alpha|^\prime}]\,\,
	a^{\alpha}. 
  \end{gather*}
\end{thm}

\begin{remark}
 After correcting a couple of typographical errors, this is a simplified
 version of the formula given in \cite[VIII.7.8]{BMMS86}.
\end{remark}

\noindent
For $n \neq p^i - 1$, the power series $MC_n(\xi)$ are McClure's
obstructions to the existence of $\hi$ structure on Quillen's map
$r\colon
MU \to BP$.  Note that, if $i + 1$ is not a power of $p$ then
$r_*[\cp{i}] = 0$, so many of the summands on $MC_n$ are zero.  For
our calculations, we make use of the following alternate expression:
\begin{prop}\label{prop:MCn-rearranged}
  McClure's formula is equivalent to
  \[
  MC_n(\xi) = \chi^{2n+1} \sum_{k = 0}^n r_*[\mathbb{C}P^{n-k}] \cdot \left({\textstyle \sum_{i \ge 0}} a_i z^i \right)^{-(n+1)}[z^k].
  \]
\end{prop}
\begin{proof}
  We rearrange the sum by summing over $|\alpha|^\prime=k$. Now the
  condition $|\alpha| = n$ is simply a constraint on $\alpha_0$.
  \begin{align*}
  MC_{n}(\xi) & = \sum_{k =
    0}^n\sum_{\substack{|\alpha|^\prime=k \\ |\alpha| = n}}\mu(-(n+1);\overline{\alpha})\,\,r_{*}[\mathbb{C}P^{n-|\alpha|^\prime}]\,\,
  a^{\alpha} \\
  & = \sum_{k = 0}^n r_{*}[\mathbb{C}P^{n-k}]\, \sum_{\substack{|\alpha|^\prime=k \\ |\alpha| = n}}\mu(-(n+1);\overline{\alpha})\,\,
  a^{\alpha}.
  \end{align*}
  To simplify the inner sum, we consider the following formal series
  and use the definition of the modified multinomial coefficients
  given in \prettyref{eq:def-mu}:
  \begin{align*}
  a_0^{2n+1} \left( {\textstyle \sum_{i \ge 0}} a_i z^i \right)^{-(n+1)} & = 
  a_0^{n} \left( 1 + \frac{a_1}{a_0}z + \frac{a_2}{a_0}z^2 +\cdots \right) ^{-(n+1)}\\
  & = a_0^n \sum_{\overline{\alpha}} \mu(-(n+1);\overline{\alpha}) \: 
  \left( \frac{a_1}{a_0}z \right)^{\alpha_1} \: \left(\frac{a_2}{a_0}z^2 \right)^{\alpha_2}\cdots\\
  & = \sum_{\overline{\alpha}} \mu(-(n+1);\overline{\alpha}) \: 
  \frac{a_0^n \: a_1^{\alpha_1}a_2^{\alpha_2}\cdots}{a_0^{\alpha_1 + \alpha_2 + \cdots}} 
    \: z^{\alpha_1 + 2 \alpha_2 + \cdots}\\
  & = \sum_{k \ge 0} z^k \left( \sum_{|\overline{\alpha}|'=k} \mu(-(n+1);\overline{\alpha}) \: 
    \frac{a_0^n \: a_1^{\alpha_1} \: a_2^{\alpha_2} \cdots}{a_0^{\alpha_1 + \alpha_2 + \cdots}} \right)\\
  \end{align*}

  Now we consider the coefficients of $z^k$.  For $k \leq n$, the
  restriction 
  $|\overline{\alpha}|' = k$
  implies
  $|\overline{\alpha}| \leq n$.
    Hence we may extend to a
  sum over multi-indices $\alpha = (\alpha_0, \alpha_1, \alpha_2,
  \ldots)$ with $\alpha_0 = n - |\overline{\alpha}|$ which forces
  $|\alpha| = n$.  Thus we have, for $0 \leq k \leq n,$

  \[
  a_0^{2n+1} \left( {\textstyle \sum_{i \ge 0}} a_i z^i \right)^{-(n+1)} [z^k] =
  \sum_{\substack{|\alpha|' = k \\ |\alpha| = n}} \mu(-(n+1);\overline{\alpha}) \: 
    a^{\alpha}.
  \]
\end{proof}

\subsection{Sparseness}\label{sec:Sparseness}
In this section we prove that, at odd primes, many of the $MC_n$ do
in fact vanish.  We also give a sparseness result for the $a_i$.

\begin{prop}\label{prop:sparseness}
  If $n \not\equiv 0 \mod{p-1}$ then $MC_n=0$.
\end{prop}
\begin{proof}
  The statement is vacuously true at the prime 2, so assume
  $p$ is odd.  The summands of the equation in
  \prettyref{thm:McClure-formula} are constant multiplies of
  $r_*[\cp{i}]$ and $a^\alpha$. The first term is nonzero only in
  degrees divisible by $2(p-1)$ and it follows from the lemma below
  that the nonzero $a^\alpha$ are also concentrated in degrees
  divisible by $2(p-1)$.

  Now the left side of the equation in \prettyref{thm:McClure-formula}
  is in degree $2n(p-2)$ and the right-hand side is concentrated in
  degrees divisible by $2(p-1)$.  Since 2 and $(p-2)$ are units mod $p$
  we see that $MC_n$ can only be non-zero when $n$ is divisible by
  $p-1$.
\end{proof}

\begin{lem}\label{lem:a_i-sparseness}
  The elements $a_i\in BP^*(BC_p)$ defined in
  \prettyref{eq:def-aseries} are zero if $i\not\equiv 0 \mod{p-1}$.
\end{lem}
\begin{proof}
  Since the lemma is vacuously true for $p=2,$ we will assume $p$ is
  odd.  

  The action of $C_p^{\times}$ on $C_p$ induces an action of
  $C_{p}^{\times}$ on $BC_p$.
  In $BP^*(BC_p),$ an element $v\in C_p^\times$ acts on $[i]\xi$ by \[
  [i]\xi\mapsto [v i]\xi.\]   Since the product 
  \[\prod_{i=1}^{p-1}([i]\xi+_{BP}x)\] is invariant under this action,
  we see that $a_i\in BP^{2(p-i-1)}(BC_p)^{C_{p}^\times}$.

  The Atiyah-Hirzebruch spectral sequence computing $BP^*(BC_p)$
  collapses at the $E_2$ page, which is of the form $H^*(BC_p,BP^*)$.
  The group action above induces a group action on this page.  Since
  the edge homomorphism $BP^*(BC_p)\rightarrow H^*(BC_p),$ is an
  equivariant surjection that restricts to an isomorphism along the
  0th row, the associated graded of $BP^*(BC_p)^{C_p^\times}$ is
  isomorphic to $H^*(BC_p)^{C_p^\times}\otimes BP^*\cong \mathbb{Z}/p[
  \xi^{p-1}]\otimes BP^*$.

  Since this last group is concentrated in degrees divisible by
  $2(p-1),$ if $a_i\neq 0$ then 
  \[a_i\in BP^{2(p-1)*}(BC_p).\]  The congruence 
  \[ 
  \frac{|a_i|}{2} =(p-1-i) \equiv i \equiv 0 \mod{(p-1)}
  \] 
  implies $i$ is divisible by $p-1$.
 
\end{proof}

As a result, it is of interest to consider $MC_{2(p-1)}$.  In this case,
one can give the formula more explicitly:
\begin{align}
  MC_{2(p-1)}(\xi) & = a_0^{2p-4}r_*[\cp{(p-1)}] \left(-(2p-1)a_0a_{(p-1)}\right)\\
  & +a_0^{2p-4}r_*[\cp{0}] \left(-(2p-1)a_0a_{2(p-1)} + p(2p-1)a_{(p-1)}^2 \right)\notag
  \intertext{Making the simplifications $[\cp{0}] = 1$ and $r_*[\cp{p-1}] = v_1$, we have}
  MC_{2(p-1)}(\xi) & = (2p-1)a_0^{2p-4}\left(-v_1a_0a_{(p-1)}-a_0a_{2(p-1)} + pa_{(p-1)}^2 \right)\notag
\end{align}

\section{Calculations}\label{sec:Computer-Generated-Calculations}
In this section, we outline the computation of the $MC_n$, work through an
example at the prime 2, and display results at the primes $p \leq \maxp$.
We have developed a Sage package \cite{JoN10} to automate the calculations.

\subsection{Description of calculation}
We are working in
$BP^*\llbracket \xi \rrbracket /\langle p \rangle \xi$, and we
emphasize reduction modulo $\langle p \rangle \xi$ by writing
$\equiv$ mod $\langle p \rangle \xi$ instead of equality.  Our
calculations have three parameters: the prime, $p$, the value of
$n$, and a truncation number, $k$. All of our computations are
modulo $(\xi,x)^{k+1}$. If power series $f(\xi)$ and $g(\xi)$ are
equal modulo the ideal $(\xi)^{k+1}$, we write
\[
f(\xi)=g(\xi)+O(\xi)^{k+1}.
\] 
It is important to note, because of this choice, that the range of
accurate coefficients for the $a_i(\xi)$ decreases as $i$ grows.
Each $a_i$ is accurate modulo $\xi^{k-i+1}$.  Using the formula
above, and the fact that $a_0 = (p-1)!\cdot\xi^{p-1} + \cdots$, we
see that $MC_{2(p-1)}$ is accurate modulo $\xi^{k-p+2}$.

We have made efforts to streamline the computation, but our results are
limited by the computational complexity of formal group law calculations.
Determining the series $\exp_{BP}$ is already a task whose computation time
grows quickly with the length of the input.  Calculating the $a_i$ is also
a high-complexity task, and as a result we do not expect direct computation
to be a feasible approach for large primes.  We have not been able to work
in a large enough range to detect non-zero values of $MC_n$ for primes
greater than \maxp.

  To check for triviality modulo $\langle p \rangle \xi$, we make use of
  the following reduction algorithm: Suppose
  $g \in (\xi)^m \subset BP^* \llbracket \xi \rrbracket$ and write
  \[
  g = \sum_{i \ge 0} g_i \xi^{i+m}
  \]
  with $g_i \in BP^*$ and $g_0 \neq 0$.  If $p \not | \ g_0$ then $g \not \in
  (\langle p \rangle \xi )$.  If $g_0 = p \cdot d_0$ for $d_0 \in BP^*,$ then
  we have
  \[
  g'(\xi) = g(\xi) - d_0 \xi^m \langle p \rangle \xi
  \]
  with $g' \in (\xi)^{m+1}$ and
  $g \equiv g' \text{ mod }\langle p \rangle \xi$.  Iterating this process
  converges in the $\xi$-adic topology.

A similar adaptation of the usual Euclidean algorithm for division by $p$
gives the following.  We state the result integrally since
we are working with the Hazewinkel generators throughout.
\begin{prop}[Division Algorithm]\label{prop:div-alg}
  Let $g$ be a power series in
  $\bbZ[v_1,v_2,\ldots] \llbracket \xi \rrbracket$ and let
  $\langle p \rangle \xi$ be the reduced $p$-series, computed using the
  Hazewinkel generators.  Then there are unique power series $d$ and
  $s = \sum_{i \ge 0} s_i \xi^i$ in
  $\bbZ[v_1,v_2,\ldots] \llbracket \xi \rrbracket$ such that
 \[
  g(\xi) = d \cdot \langle p \rangle \xi + s
  \]
  and such that the polynomials $s_i \in \bbZ[v_1,v_2,\ldots]$
  have coefficients in the range $\{0, \ldots, p-1\}$.  The series $g$ is divisible by $\langle p \rangle \xi$ if and only
  if $s = 0$.
\end{prop}

\subsection{\texorpdfstring{Sample calculation, $p = 2$}{Example calculation, p = 2}}\label{eg:eg-calc-2}
To give the reader a sense of how these calculations are implemented,
we work through the calculation of $MC_2(\xi)$ with the minimum range
of coefficients necessary to see that it is non-zero.  For this, it is
necessary to work modulo $(x,\xi)^8$. The formula for $MC_2$ is given
in \prettyref{prop:MCn-rearranged}:
\begin{align*}
  MC_2(\xi) & = a_0^{5} \sum_{k = 0}^2 r_*[\mathbb{C}P^{n-k}] \cdot
  \left({\textstyle \sum_{i \ge 0}} a_i z^i \right)^{-(n+1)}[z^k].
\end{align*}

Now one can easily check the formal computation
\begin{align*}
\left({\textstyle \sum_{i \ge 0}} a_i z^i \right)^{-1} = & \ 
a_0^{-1}
-a_1 a_0^{-2} z
+\left(-a_2 a_0^{-2}+a_1^2 a_0^{-3}\right) z^2\\
&+O(z)^3
\intertext{and hence}
\left({\textstyle \sum_{i \ge 0}} a_i z^i \right)^{-3} = & \ 
a_0^{-3}
-3a_1 a_0^{-4} z
+\left(-3a_2 a_0^{-4}+6a_1^2 a_0^{-5}\right) z^2\\
&+O(z)^3.
\end{align*}
The image of $[\cp{i}] \in MU^{-2i}$ under $r_*$ is given by
\[
r_*[\cp{i}] = 
\begin{cases}
0 & \text{if~} i \neq p^k - 1\\
[\cp{i}] = p^k\ell_k & \text{if~} i = p^k - 1\\
\end{cases}
\]
The elements $\ell_k$ are rational generators for $BP$, but it is
convenient to work with integral generators. For this example we
choose the Hazewinkel generators $v_k$, but the result is independent
of this choice.  It will be necessary only to use $v_1 = 2\ell_1$, so
we work modulo the ideal $I = (v_2, v_3, \ldots)$.  Modulo $I$ we
have $4\ell_2 = v_1^3$, and this will be the only additional
substitution we need to use.


Returning to the calculation, we have
\[
[\cp{0}] = 1,
\quad
r_*[\cp{1}] = 2\ell_1 = v_1,
\quad \text{and }
r_*[\cp{2}] = 0
\]
and so
\begin{align*}
MC_2(\xi) & = a_0^5 \left( -3v_1 a_0^{-4}a_1 +(-3a_2 a_0^{-4}+6a_1^2 a_0^{-5}) \right)\\
& = 6a_1^2 - 3a_0a_2 - 3v_1 a_0 a_1.
\end{align*}
To continue, we determine $a_0(\xi),$ $a_1(\xi),$ and $a_2(\xi)$.
These are defined by the following
(cf. \ref{eq:power-op-formula}, \ref{eq:def-aseries}):
\begin{align*}
 P_{C_p,BP}(x) = & \ r_*P_{C_p,MU}(x) = \prod_{i = 0}^1 \left([i]\xi +_{BP} x\right) = x \cdot \exp\left( \log(\xi) + \log(x) \right)\\
= & \ x \,\cdot \left[ 
a_0
+a_1 x^1
+a_2 x^2
+a_3 x^3 \right.\\
&\qquad +a_4 x^4
+a_5 x^5
+a_6 x^6
+a_7 x^7\\
&\qquad \left. +O(x,\xi)^8 \ \right].\\
\end{align*}
The logarithm is
\[
\log_{BP}(\xi) = \xi + \ell_1 \xi^2 + \ell_2 \xi^4 
+ O(\xi)^8
\]
and hence the exponential is
\begin{align*}
\exp_{BP}(\xi) = &\ 
\xi
-\ell _1 \xi^2
+2 \ell _1^2 \xi^3
+\left(-5 \ell _1^3-\ell _2\right) \xi^4\\
&+\left(14 \ell _1^4+6 \ell _1 \ell _2\right) \xi^5\\
&+\left(-42 \ell _1^5-28 \ell _1^2 \ell _2\right) \xi^6\\
&+\left(132 \ell _1^6+120 \ell _1^3 \ell _2+4 \ell _2^2\right) \xi^7\\
&+O(\xi)^8.
\end{align*}
Using the logarithm and exponential, we give the reduced $2$-series:
\begin{align*}
\langle 2 \rangle \xi = \frac{1}{\xi}\exp(2 \log(\xi)) = &\ 
2
-2 \ell _1 \xi
+8 \ell _1^2 \xi^2\\
&+\left(-36 \ell _1^3-14 \ell _2\right) \xi^3\\
&+\left(176 \ell _1^4+120 \ell _1 \ell _2\right) \xi^4\\
&+\left(-912 \ell _1^5-888 \ell _1^2 \ell _2\right) \xi^5\\
&+\left(4928 \ell _1^6+6240 \ell _1^3 \ell _2+448 \ell _2^2\right) \xi^6\\
&+O(\xi)^7
\intertext{
  Substituting the Hazewinkel generators, and working modulo $v_2$,
}
\langle 2 \rangle \xi = &\ 
2
-v_1 \xi
+2 v_1^2 \xi^2\\
&-8 v_1^3\xi^3\\
&+26 v_1^4\xi^4\\
&-84 v_1^5\xi^5\\
&+300 v_1^6\xi^6\\
&+O(\xi)^7
\end{align*}
and
\begin{align*}
 P_{C_p,BP} = x\;\cdot &\left[
\left( \xi + x + \ell_1 (\xi^2 + x^2) + \ell_2 (\xi^4 + x^4) + \ell_3 (\xi^8 +x^8) \right)\right.\\
& -\ell _1 \left( \xi + x + \ell_1 (\xi^2 + x^2) + \ell_2 (\xi^4 + x^4) + \ell_3 (\xi^8 +x^8) \right)^2\\
& +2\ell _1^2 \left( \xi + x + \ell_1 (\xi^2 + x^2) + \ell_2 (\xi^4 + x^4) + \ell_3 (\xi^8 +x^8) \right)^3\\
&+\left(-5 \ell _1^3-\ell _2\right) \left( \xi + x + \ell_1 (\xi^2 + x^2) + \ell_2 (\xi^4 + x^4) + \ell_3 (\xi^8 +x^8) \right)^4\\
&+\left(14 \ell _1^4+6 \ell _1 \ell _2\right) \left( \xi + x + \ell_1 (\xi^2 + x^2) + \ell_2 (\xi^4 + x^4) + \ell_3 (\xi^8 +x^8) \right)^5\\
&+\left(-42 \ell _1^5-28 \ell _1^2 \ell _2\right) \cdot\\
&\qquad \quad\left( \xi + x + \ell_1 (\xi^2 + x^2) + \ell_2 (\xi^4 + x^4) + \ell_3 (\xi^8 +x^8) \right)^6\\
&+\left(132 \ell _1^6+120 \ell _1^3 \ell _2+4 \ell _2^2\right) \cdot\\
&\qquad \quad\left( \xi + x + \ell_1 (\xi^2 + x^2) + \ell_2 (\xi^4 + x^4) + \ell_3 (\xi^8 +x^8) \right)^7\\
&\left.+O(x,\xi)^8\right].\\
\end{align*}
Expanding, and substituting the Hazewinkel generators, we have
\begin{align*}
a_0 & = \xi +O(\xi)^8\\
a_1 & = 1
-v_1 \xi+v_1^2 \xi^2
-2 v_1^3\xi^3\\
& \quad +3 v_1^4\xi^4
-4 v_1^5\xi^5\\
& \quad +v_1^6\xi^6
+ O(\xi)^7\\
&\equiv 1+v_1 \xi+v_1^4 \xi^4 +v_1^5 \xi^5+v_1^6\xi^6
+O(\xi)^7 \quad \text{mod } \langle 2 \rangle \xi \\
a_2 & = v_1^2 \xi
-4 v_1^3\xi^2
+10 v_1^4\xi^3
-21 v_1^5\xi^4\\
& \quad +43 v_1^6\xi^5
+ O(\xi)^6\\
& \equiv v_1^2 \xi+v_1^5 \xi^4
+O(\xi)^6 \quad \text{mod } \langle 2 \rangle \xi. \\
\end{align*}

Substituting into the formula for $MC_2$, we have (modulo $v_2$)
\begin{align*}
  MC_2(\xi) = & 6a_1^2 - a_0a_2 - 3v_1 a_0 a_1\\
  \equiv &\  
  6\left(
    1+v_1 \xi+v_1^4 \xi^4 +v_1^5 \xi^5+v_1^6\xi^6+O(\xi)^7
  \right)^2 \\
  &-3 \left(
    \xi +O(\xi)^8
  \right)
  \left(
    v_1^2 \xi+v_1^5 \xi^4+O(\xi)^6
  \right)\\ 
  &- 3v_1 \left(
    \xi +O(\xi)^8
  \right) 
  \left(
    1+v_1 \xi+v_1^4 \xi^4 +v_1^5 \xi^5+v_1^6\xi^6+O(\xi)^7
  \right)\\
  &\quad \text{ mod }\langle 2 \rangle\xi\\
  = &\  6+9 v_1 \xi+12 v_1^4\xi^4+18 v_1^5 \xi^5+21 v_1^6\xi^6
  +O(\xi)^7 \quad\text{mod }\langle 2 \rangle \xi.
  \intertext{
Note that, although $a_2$ is accurate only modulo $\xi^6$, the product
$a_0a_2$ is accurate modulo $\xi^7$ and hence $MC_2$ is accurate
modulo $\xi^7$.
Since the lowest-order term is $3\cdot2$, we subtract $3\cdot\langle 2 \rangle \xi$ to give}
  MC_2(\xi) \equiv &\ 
  12 v_1 \xi-6 v_1^2 \xi^2+v_1^3\xi^3-66 v_1^4\xi^4+270 v_1^5\xi^5-879 v_1^6\xi^6+O(\xi)^7 \quad\text{mod }\langle 2 \rangle \xi.
\end{align*}
Continuing to reduce in this way gives the following:
\[
  MC_2(\xi) \equiv \ v_1^6 \xi^6 + O(\xi)^7 \qquad \text{mod }\langle 2 \rangle \xi.
\]
Since the lowest-order term of the right-hand side is non-zero mod 2,
the entire expression is non-zero in $BP^*\llbracket \xi
\rrbracket/\langle 2 \rangle \xi$.

\subsection{\texorpdfstring{Results at $p = 2$}{Results at p = 2}}

\begin{align*}
  \langle 2 \rangle \xi = \; & 2 -\xi  v_1+2 \xi ^2 v_1^2+\xi ^3 \left(-8 v_1^3-7 v_2\right)+\xi ^4 \left(26 v_1^4+30 v_1 v_2\right)\\
  &+\xi ^5 \left(-84 v_1^5-111 v_1^2 v_2\right)+\xi ^6 \left(300 v_1^6+502 v_1^3 v_2+112 v_2^2\right)\\
  &+\xi ^7 \left(-1140 v_1^7-2299 v_1^4 v_2-960 v_1 v_2^2-127 v_3\right)\\
  &+\xi ^8 \left(4334 v_1^8+9958 v_1^5 v_2+5414 v_1^2 v_2^2+766 v_1 v_3\right)\\
  &+\xi ^9 \left(-16692 v_1^9-43118 v_1^6 v_2-29579 v_1^3 v_2^2-2380 v_2^3
    -3579 v_1^2 v_3\right)\\
  &+\xi ^{10} \left(65744 v_1^{10}+189976 v_1^7 v_2+161034 v_1^4 v_2^2+31012 v_1 v_2^3+17770 v_1^3 v_3+5616 v_2 v_3\right)\\
  &+\xi ^{11} \left(-262400 v_1^{11}-837637 v_1^8 v_2-838452 v_1^5 v_2^2-240631 v_1^2 v_2^3-86487 v_1^4 v_3\right.\\
  &\quad \qquad \left.-55329 v_1 v_2 v_3\right)\\
  &+\xi ^{12} \left(1056540 v_1^{12}+3685550 v_1^9 v_2+4232750 v_1^6
    v_2^2+1600786 v_1^3 v_2^3+58268 v_2^4 \right.\\
  &\quad \qquad \left.+404198 v_1^5 v_3+363210 v_1^2v_2 v_3\right)\\
  &+\xi ^{13} \left(-4292816 v_1^{13}-16254540 v_1^{10} v_2-21110372
    v_1^7 v_2^2-10071369 v_1^4 v_2^3-1022466 v_1 v_2^4\right.\\
  & \quad \qquad \left.-1864478 v_1^6 v_3-2193009 v_1^3 v_2 v_3-212440 v_2^2 v_3\right)\\
  &+O(\xi)^{14}\\
  \\
  MC_1(\xi) \equiv \; & \xi ^2 v_1^2 +\xi ^3 v_2 +\xi ^4
  \left(v_1^4+v_1 v_2\right) +\xi ^7 \left(v_1^7+v_3\right)
  +\xi ^8 \left(v_1^8+v_1 v_3\right) \\
  &+\xi ^9 \left(v_1^9+v_1^6 v_2+v_1^3 v_2^2+v_2^3+v_1^2 v_3\right)
  +\xi ^{10} \left(v_1^{10}+v_1 v_2^3+v_1^3 v_3\right)
  +\xi ^{11} \left(v_1^5 v_2^2+v_1 v_2 v_3\right)\\
  &+\xi ^{12} \left(v_1^{12}+v_1^9 v_2+v_1^6 v_2^2+v_1^3
    v_2^3+v_2^4+v_1^5 v_3\right)
  +\xi ^{13} v_1^4 v_2^3\\
  & +O(\xi)^{14} \quad \text{mod } \langle 2 \rangle \xi\\
  \\
  MC_2(\xi) \equiv \; & \xi ^6 \left(v_1^6+v_2^2\right) +\xi ^7
  \left(v_1^7+v_3\right) +\xi ^8 \left(v_1^5 v_2+v_1 v_3\right) +\xi
  ^9 v_2^3
  +\xi ^{10} \left(v_1^4 v_2^2+v_1 v_2^3\right)\\
  &+\xi ^{11} \left(v_1^5 v_2^2+v_1^2 v_2^3+v_1^4 v_3\right) +\xi
  ^{12} \left(v_1^9 v_2+v_1^5 v_3\right)
  +\xi ^{13} \left(v_1^{13}+v_1^{10} v_2+v_1^3 v_2 v_3\right)\\
  & +O(\xi)^{14} \quad \text{mod } \langle 2 \rangle \xi\\
  \\
  MC_3(\xi) \equiv \; & \xi ^6 v_1^6 +\xi ^7 \left(v_1^4 v_2+v_1
    v_2^2\right) +\xi ^8 \left(v_1^8+v_1^5 v_2+v_1 v_3\right) +\xi
  ^{10} \left(v_1^{10}+v_1^7 v_2+v_1^4 v_2^2+v_1^3 v_3+v_2
    v_3\right)\\
  &+\xi ^{11} \left(v_1^{11}+v_1^8 v_2+v_1^4 v_3+v_1 v_2 v_3\right)
  +\xi ^{12} v_1^3 v_2^3
  +\xi ^{13} \left(v_1^{13}+v_1^3 v_2 v_3+v_2^2 v_3\right)\\
  & +O(\xi)^{14} \quad \text{mod } \langle 2 \rangle \xi\\
  \\
  MC_4(\xi) \equiv \; & \xi ^{10} v_1^4 v_2^2
  +\xi ^{11} \left(v_1^{11}+v_1^8 v_2+v_1^5 v_2^2+v_1^4 v_3\right)\\
  &+\xi ^{12} \left(v_1^9 v_2+v_1^3 v_2^3+v_2^4\right) +\xi ^{13}
  \left(v_1^{10} v_2+v_1^4 v_2^3+v_1^6 v_3+v_1^3 v_2
    v_3+v_2^2 v_3\right)\\
  &+O(\xi)^{14} \quad \text{mod } \langle 2 \rangle \xi\\
  \\
  MC_5(\xi) \equiv \; & 0 + O(\xi)^{14} \quad \text{mod } \langle 2
  \rangle \xi
\end{align*}

\subsection{\texorpdfstring{Results at $p = 3$}{Results at p = 3}}

\begin{align*}
  \langle 3 \rangle \xi = \; & 3 -8 \xi ^2 v_1+72 \xi ^4 v_1^2-840 \xi
  ^6 v_1^3\\ 
  &+\xi ^8 \left(9000 v_1^4-6560 v_2\right)+\xi ^{10}
  \left(-88992 v_1^5+216504 v_1 v_2\right)\\ 
  &+\xi ^{12} \left(658776
  v_1^6-5360208 v_1^2 v_2\right)+\xi ^{14} \left(1199088 v_1^7+119105576
  v_1^3 v_2\right)\\ 
  &+\xi ^{16} \left(-199267992 v_1^8-2424100032 v_1^4
  v_2+129120480 v_2^2\right)\\ 
  &+\xi ^{18} \left(5896183992
  v_1^9+45824243688 v_1^5 v_2-8307203592 v_1 v_2^2\right)\\ 
  &+\xi ^{20}
  \left(-133449348816 v_1^{10}-807801733088 v_1^6 v_2+336744805688 v_1^2
  v_2^2\right)\\
  &+\xi ^{22} \left(2658275605728 v_1^{11}+13162584394728 v_1^7 v_2-11021856839856 v_1^3 v_2^2\right)\\ 
  &+\xi ^{24}
  \left(-48579725371464 v_1^{12}-193206868503840 v_1^8
  v_2+314960186505360 v_1^4 v_2^2\right.\\
  &\quad \qquad \left.-3670852206240 v_2^3\right)\\
  &+O(\xi)^{26}\\
  \\
  MC_2(\xi) \equiv \; & 
  v_1^3 \xi ^8
  +2 v_2 \xi ^{10}
  +\left(v_1^5+v_2 v_1\right) \xi ^{12}
  +2 v_1^2 v_2 \xi ^{14}
  +2 v_1^7 \xi ^{16}
  +\left(2 v_1^8+v_2^2\right) \xi ^{18}\\
  &+\left(v_2 v_1^5+v_2^2 v_1\right) \xi ^{20}
  +\left(2 v_1^{10}+2 v_2 v_1^6+v_2^2 v_1^2\right) \xi ^{22}
  +\left(v_1^{11}+v_2 v_1^7\right) \xi ^{24}\\
  &+O(\xi)^{26} \qquad \text{mod } \langle 3 \rangle \xi\\  
\\
  MC_4(\xi) \equiv \; & 2 v_1^9 \xi ^{22}+2 v_1^{10} \xi ^{24}
  +O(\xi)^{26} \qquad \text{mod } \langle 3 \rangle \xi\\  
\end{align*}

\begin{note}
  For $p > 3,$ we omit $\langle p \rangle \xi$ and $MC_{(p-1)}$.
\end{note}
\ 

\subsection{\texorpdfstring{Results at $p = 5$}{Results at p = 5}}

\begin{align*}
MC_8 (\xi) \equiv \; & 3v_1^{16}\xi^{88} + (4v_1^{17} + v_1^{11}v_2)\xi^{92} +
(3v_1^{18} + 4v_1^6v_2^2)\xi^{96}  + O(\xi^{100})   \quad \text{mod } \langle 5 \rangle \xi\\
\end{align*}

\subsection{\texorpdfstring{Results at $p = 7$}{Results at p = 7}}

\begin{align*}
MC_{12}(\xi) \equiv \; & 4v_1^{22}\xi^{192} + (4v_1^{23} +
2v_1^{15}v_2)\xi^{198} + (6v_1^{24} + 4v_1^{16}v_2 +
5v_1^8v_2^2)\xi^{204}\\
& + (5v_1^{25} + 5v_1^{17}v_2 + 4v_1^9v_2^2
+ 3v_1v_2^3)\xi^{210} + (2v_1^{18}v_2 + 3v_1^{10}v_2^2 +
4v_1^2v_2^3)\xi^{216} \\
& +O(\xi^{222})   \quad \text{mod } \langle 7 \rangle \xi\\
\end{align*}

\subsection{\texorpdfstring{Results at $p = 11$}{Results at p = 11}}

\begin{align*}
MC_{20}(\xi) \equiv \; & 9v_1^{34}\xi^{520} + (8v_1^{35} +
6v_1^{23}v_2)\xi^{530} + (7v_1^{36} + v_1^{24}v_2 +
5v_1^{12}v_2^2)\xi^{540}  + O(\xi^{550})   \quad \text{mod } \langle 11 \rangle \xi\\
\end{align*}

\subsection{\texorpdfstring{Results at $p = 13$}{Results at p = 13}}

\begin{align*}
MC_{24}(\xi) \equiv \; & 11v_1^{40}\xi^{744} + (6v_1^{41} +
6v_1^{27}v_2)\xi^{756} + O(\xi^{768})   \quad \text{mod } \langle 13 \rangle \xi\\
\end{align*}

\bibliographystyle{amsalpha}

\bibliography{Johnson-Noel_TOPOL-4016.bbl}

\end{document}